\documentclass[a4paper, 10pt]{amsart}
\usepackage[all,cmtip]{xy}      
\usepackage{mathtools, amsthm, amsfonts, amssymb}
\usepackage{enumitem}  
\setenumerate{label={\normalfont(\alph*)}}

\unitlength=1mm \linethickness{0.5pt}  

\newtheorem{theorem}{Theorem}[section]
\newtheorem{lemma}[theorem]{Lemma}
\newtheorem{proposition}[theorem]{Proposition}
\newtheorem{corollary}[theorem]{Corollary}

\theoremstyle{definition}

\newtheorem{example}[theorem]{Example}

\theoremstyle{remark}
\newtheorem*{remark}{Remark}

\begin{document}

\title[Connected sums and minimally non-Golod complexes]{Connected sums of sphere products and minimally non-Golod complexes} 
\author{Steven Amelotte}
\address{}
\email{steven.amelotte@rochester.edu} 


\begin{abstract}
We show that if the moment-angle complex $\mathcal{Z}_K$ associated to a simplicial complex $K$ is homotopy equivalent to a connected sum of sphere products with two spheres in each product, then $K$ decomposes as the simplicial join of an $n$-simplex~$\Delta^n$ and a minimally non-Golod complex. In particular, we prove that $K$ is minimally non-Golod for every moment-angle complex $\mathcal{Z}_K$ homeomorphic to a connected sum of two-fold products of spheres, answering a question of Grbi\'c, Panov, Theriault and Wu.
\end{abstract}

\maketitle

\section{Introduction}

A central construction in toric topology functorially assigns to each finite simplicial complex $K$ on $m$ vertices a finite $CW$-complex $\mathcal{Z}_K$, called the \emph{moment-angle complex}, which comes equipped with a natural action of the $m$-torus $T^m=(S^1)^m$. Various homological invariants of Stanley--Reisner rings of basic importance in combinatorial commutative algebra are given geometric realizations by $\mathcal{Z}_K$ and related spaces. For example, the homotopy orbit space of $\mathcal{Z}_K$ is the \emph{Davis--Januszkiewicz space} whose cohomology (with coefficients in a commutative ring $\mathbf{k}$) is the Stanley--Reisner ring $\mathbf{k}[K]$ itself, while the ordinary cohomology of $\mathcal{Z}_K$ recovers its Koszul homology (cf. \cite{BBP}, \cite{F}):
\[ H^*(\mathcal{Z}_K; \mathbf{k}) \cong \mathrm{Tor}_{\mathbf{k}[v_1,\ldots,v_m]}^\ast(\mathbf{k}[K],\mathbf{k}). \]

Combinatorial properties of simplicial complexes and, in particular, homological properties of their Stanley--Reisner rings can therefore be studied by investigating the homotopy types of moment-angle complexes. This point of view has recently been useful in establishing the Golod property for $\mathbf{k}[K]$ for certain families of simplicial complexes by applying homotopy theoretic methods to show that the corresponding moment-angle complex $\mathcal{Z}_K$ is homotopy equivalent to a wedge of spheres (see e.g. \cite{GPTW}, \cite{GT}, \cite{IK}, \cite{IK2}). Here, $\mathbf{k}[K]$ is \emph{Golod} if all products and higher Massey products vanish in $\mathrm{Tor}_{\mathbf{k}[v_1,\ldots,v_m]}^\ast(\mathbf{k}[K],\mathbf{k})$. (Golodness for a graded or local ring implies that its Poincar\'e series is a rational function, and an equivalent definition can be given in terms of a certain equality of formal power series; cf. \cite{G}.) A simplicial complex $K$ is called \emph{Golod} if $\mathbf{k}[K]$ is Golod for every field $\mathbf{k}$.

Berglund and J\"ollenbeck observed in \cite{BJ} that the Golod property is stable under deletion of vertices and introduced the notion of a \emph{minimally non-Golod} complex, that is, a non-Golod  simplicial complex which becomes Golod after deleting any of its vertices. Using combinatorial arguments, they showed that the boundary complexes of stacked polytopes are minimally non-Golod. Further examples have been given by Limonchenko \cite{L}, who showed that the nerve complexes of even dimensional dual neighbourly polytopes and certain generalized truncation polytopes are minimally non-Golod. In each of these cases, the corresponding moment-angle complex $\mathcal{Z}_K$ is well known to be a smooth manifold diffeomorphic to a connected sum of sphere products with two spheres in each product (see \cite{BM} and \cite{GL}). Moreover, in \cite{GPTW} it was shown that if $K$ is a flag complex, then $K$ being minimally non-Golod is equivalent to the condition that $\mathcal{Z}_K$ is a connected sum of two-fold products of spheres. The authors raised the question of whether, more generally, $K$ is minimally non-Golod for all simplicial complexes for which $\mathcal{Z}_K$ has such a diffeomorphism type (\cite[Question 3.5]{GPTW}). The purpose of the present paper is to answer this question affirmatively.

\begin{theorem} \label{homeomorphic}
If $\mathcal{Z}_K$ is homeomorphic to a connected sum of sphere products with two spheres in each product, then $K$ is minimally non-Golod.
\end{theorem}

\begin{remark}
The statement of Theorem \ref{homeomorphic} is not true if the ``homeomorphic" condition is replaced by ``homotopy equivalent". In Section 3 we give a counterexample in the form of a cone over a minimally non-Golod complex $K$ for which $\mathcal{Z}_K$ is homotopy equivalent to a connected sum of sphere products (see Example \ref{example}), and we prove that iterated cones of this type are the only such counterexamples. We also remark that the converse of Theorem \ref{homeomorphic} is not true. In \cite{L2}, stellar subdivisions of minimal triangulations of $T^2$ and $\mathbb{C}P^2$ are shown to be minimally non-Golod complexes whose corresponding moment-angle complexes are not connected sums of sphere products.
\end{remark}

We give two proofs of Theorem \ref{homeomorphic}. The first is a direct proof that makes crucial use of the assumption that the moment-angle complex $\mathcal{Z}_K$ has the structure of a closed manifold. The second comes as a corollary of the slightly more general Theorem \ref{homotopy equivalent} below.

For $n\ge-1$, let $\Delta^n$ be the standard $n$-simplex, where $\Delta^{-1}=\varnothing$ is the empty simplicial complex.

\begin{theorem} \label{homotopy equivalent}
If $\mathcal{Z}_K$ is homotopy equivalent to a connected sum of sphere products with two spheres in each product, then $K = \Delta^n \ast L$ for some $n\ge -1$ where $L$ is Gorenstein* and minimally non-Golod. 
\end{theorem}

Let $K$ be a simplicial complex on the vertex set $[m]$. The \emph{star} of a vertex $v \in K$ is the subcomplex \[ \mathrm{star}_K(v) = \{ \sigma \in K \mid \{v\} \cup \sigma \in K \}. \]
The \emph{core} of $K$ is then defined to be the full subcomplex \[ \mathrm{core}(K) = K_{\{v\in [m] \,\mid\, \mathrm{star}_K(v) \neq K\}} \] of $K$ on the restricted vertex set $\{v\in [m] \mid \mathrm{star}_K(v) \neq K\}$. Note that any simplicial complex can be written as a join 
\begin{equation} \label{join}
K=\Delta^n \ast \mathrm{core}(K),
\end{equation} 
where $\Delta^n$ is the simplex with (possibly empty) vertex set $[m] \backslash \mathrm{core}(K) = \{v\in [m] \mid \mathrm{star}_K(v) = K\}$.
Since the moment-angle complex functor carries simplicial joins to Cartesian products and $\mathcal{Z}_{\Delta^n} \cong D^{2(n+1)}$ is contractible, it follows from \eqref{join} that the homotopy type of a moment-angle complex $\mathcal{Z}_K$ is determined by the core of $K$. In the notation of Theorem~\ref{homotopy equivalent}, it will be shown that $L=\mathrm{core}(K)$ and hence that any simplicial complex satisfying the hypothesis of Theorem~\ref{homotopy equivalent} has a minimally non-Golod core.
The Gorenstein* property implies that $\mathcal{Z}_{\mathrm{core}(K)}$ is a closed orientable manifold. 


The author would like to thank the Fields Institute for Research in Mathematical Sciences and the organizers of the Thematic Program on Polyhedral Products and Toric Topology for providing both an excellent setting for research and an opportunity to present this work. In particular, comments and questions from Taras Panov, Don Stanley and Stephen Theriault led to improvements in this paper.

\section{Preliminaries}

Throughout this paper, $K$ will denote a finite abstract simplicial complex on the vertex set $[m]=\{1,\ldots,m\}$. We always assume that $\varnothing \in K$ and that $K$ has no ghost vertices, that is, $\{i\}\in K$ for all $i=1,\ldots,m$.

Let $(\underline{X},\underline{A}) = \{(X_i,A_i)\}_{i=1}^m$ be a sequence of pointed $CW$-pairs. For each simplex $\sigma \in K$, define $(\underline{X},\underline{A})^\sigma$ to be the subspace of $\prod_{i=1}^m X_i$ given by
\[ (\underline{X},\underline{A})^\sigma = \{(x_1,\ldots,x_m) \in \textstyle\prod_{i=1}^m X_i \mid x_i \in A_i \text{ for } i \notin \sigma\}. \]
The \emph{polyhedral product} of $(\underline{X},\underline{A})$ corresponding to $K$ is then defined by
\begin{equation} \label{pp definition}
(\underline{X},\underline{A})^K = \bigcup_{\sigma\in K} (\underline{X},\underline{A})^\sigma  \subseteq \prod_{i=1}^m X_i. 
\end{equation}

In the case where $(X_i,A_i)=(D^2,S^1)$ for each $i=1,\ldots,m$, the polyhedral product corresponding to $K$ is called the \emph{moment-angle complex}, denoted $\mathcal{Z}_K$. Similarly, the \emph{real moment-angle complex} $\mathcal{R}_K$ is defined by the polyhedral product $(\underline{X},\underline{A})^K$ with $(X_i,A_i)=(D^1,S^0)$ for each $i=1,\ldots,m$. Generalizing these two cases of special interest, much of the work to date on the homotopy theory of polyhedral products has focused on pairs of the form $(CX_i,X_i)$, where $CX_i$ is the reduced cone on $X_i$. For a sequence of spaces $\underline{X}=\{X_i\}_{i=1}^m$, let $C\underline{X}=\{CX_i\}_{i=1}^m$.

For $I\subseteq [m]$, the \emph{full subcomplex} of $K$ on the vertex set $I$ is defined by \[K_I=\{\sigma\in K \mid \sigma\subseteq I\}. \]
For any vertex $i\in [m]$, we denote by $K-\{i\}$ the \emph{deletion complex} of $i$ defined by
\[ K-\{i\} = \{\sigma \in K \mid i \notin \sigma\}. \] 
Note that $K-\{i\}$ is the full subcomplex of $K$ on the restricted vertex set $[m]\backslash \{i\}$. We will need the following basic but useful property of polyhedral products associated to full subcomplexes.

\begin{proposition} \label{retract}
Let $K$ be a simplicial complex on the vertex set $[m]$ and let $I \subseteq [m]$ be a non-empty subset. Then $(\underline{X},\underline{A})^{K_I}$ is a retract of $(\underline{X},\underline{A})^K$.
\end{proposition}

\begin{proof}
Let $I=\{i_1,\ldots,i_k\} \subseteq [m]$ where $1\le i_1\le\cdots\le i_k\le m$ and $k\ge 1$. The simplicial inclusion $K_I \longrightarrow K$ induces a map of polyhedral products $j_I\colon (\underline{X},\underline{A})^{K_I} \longrightarrow (\underline{X},\underline{A})^K$. It is straightforward to check that the projection $ \prod_{j=1}^m X_j \longrightarrow \prod_{j=1}^k X_{i_j}$ restricts to a map $r\colon (\underline{X},\underline{A})^K \longrightarrow (\underline{X},\underline{A})^{K_I}$ and that the composite 
\[ (\underline{X},\underline{A})^{K_I} \stackrel{j_I}{\longrightarrow} (\underline{X},\underline{A})^K \stackrel{r}{\longrightarrow} (\underline{X},\underline{A})^{K_I} \]
is the identity map.
\end{proof}

Let $\widehat{(\underline{X},\underline{A})}^K$ denote the image of $(\underline{X},\underline{A})^K$ under the natural quotient map $\prod_{i=1}^mX_i \longrightarrow \bigwedge_{i=1}^mX_i$. After suspending, the retraction maps of Proposition \ref{retract} for all full subcomplexes of $K$ can be added together using the co-$H$-space structure on $\Sigma(\underline{X}, \underline{A})^K$ to obtain the following splitting due to Bahri, Bendersky, Cohen and Gitler.

\begin{theorem}[{\cite[Theorem~2.10]{BBCG}}] 
Let $K$ be a simplicial complex on the vertex set $[m]$ and let $(\underline{X},\underline{A}) = \{(X_i,A_i)\}_{i=1}^m$ be a sequence of pointed $CW$-pairs. Then there is a natural homotopy equivalence
\[ \Sigma(\underline{X}, \underline{A})^K \simeq \bigvee_{I\subseteq [m]} \Sigma \widehat{(\underline{X},\underline{A})}^{K_I}. \]
\end{theorem}

The authors of \cite{BBCG} go on to further identify the spaces appearing on the right hand side of the wedge decomposition above in various cases of interest.

\begin{theorem}[{\cite{BBCG}}] \label{BBCG splitting}
Let $K$ be a simplicial complex on the vertex set $[m]$ and let $\underline{X}=\{X_i\}_{i=1}^m$ be a sequence of pointed $CW$-complexes. Then there is a homotopy equivalence
\[
\Sigma(C\underline{X},\underline{X})^K \simeq \bigvee_{I\notin K} \Sigma^2|K_I| \wedge \widehat{X}^I 
\]
where $\widehat{X}^I = X_{i_1} \wedge \cdots \wedge X_{i_k}$ for $I=\{i_1,\ldots,i_k\}$.
\end{theorem}

In the special case where $X_i=S^1$ for all $i=1,\ldots,m$, Theorem \ref{BBCG splitting} gives the following suspension splitting for moment-angle complexes, which can be regarded as a geometric realization of the description of $\mathrm{Tor}_{\mathbf{k}[v_1,\ldots,v_m]}^\ast(\mathbf{k}[K], \mathbf{k})$ given by Hochster's Theorem.

\begin{corollary} There is a homotopy equivalence
\[ \Sigma\mathcal{Z}_K \simeq \bigvee_{I\notin K} \Sigma^{|I|+2}|K_I|. \]
\end{corollary}

Next, we use the splittings above to prove a lemma which will be needed in the proof of Theorem \ref{homotopy equivalent} to compare the homotopy types of moment-angle complexes associated to a simplicial complex $K$ and its deletion complexes $K-\{i\}$. 

\begin{lemma} \label{key lemma}
Let $K$ be a simplicial complex on the vertex set $[m]$ and let $\underline{X}=\{X_j\}_{j=1}^m$ be a sequence of pointed $CW$-complexes which are non-contractible. Let $i\in [m]$. Then the natural inclusion $(C\underline{X},\underline{X})^{K-\{i\}} \longrightarrow (C\underline{X},\underline{X})^K$ is a homotopy equivalence if and only if $K = \{i\} \ast (K-\{i\})$.
\end{lemma}

\begin{proof}
If $K = \{i\} \ast (K - \{i\})$ is the cone over the deletion complex $K-\{i\}$, then permuting coordinates defines a homeomorphism
\[ (C\underline{X},\underline{X})^K \cong CX_i \times (C\underline{X},\underline{X})^{K-\{i\}}, \]
where the sequence of pairs of spaces $(C\underline{X},\underline{X})$ on the right-hand side is understood to be $\{(CX_j,X_j)\}_{j\in [m]\backslash\{i\}}$. The natural inclusion $(C\underline{X},\underline{X})^{K-\{i\}} \longrightarrow (C\underline{X},\underline{X})^K$ composed with the homeomorphism above is the inclusion of the right-hand factor in the product $CX_i \times (C\underline{X},\underline{X})^{K-\{i\}}$, which is a homotopy equivalence since $CX_i$ is contractible.

Conversely, suppose $(C\underline{X},\underline{X})^{K-\{i\}} \longrightarrow (C\underline{X},\underline{X})^K$ is a homotopy equivalence. By Theorem \ref{BBCG splitting}, there is a suspension splitting
\begin{align*}
\Sigma(C\underline{X},\underline{X})^K &\simeq \bigvee_{I\notin K} \Sigma^2|K_I| \wedge \widehat{X}^I \\
&\simeq \Bigg( \bigvee_{\substack{I\notin K \\ i\notin I}} \Sigma^2|K_I| \wedge \widehat{X}^I \Bigg) \vee \Bigg( \bigvee_{\substack{I\notin K \\ i\in I}} \Sigma^2|K_I| \wedge \widehat{X}^I \Bigg) \\
&\simeq \quad \Sigma(C\underline{X},\underline{X})^{K-\{i\}} \hspace{1.05em} \vee \Bigg( \bigvee_{\substack{I\notin K \\ i\in I}} \Sigma^2|K_I| \wedge \widehat{X}^I \Bigg),
\end{align*} 
and, up to homotopy, the suspended inclusion $\Sigma(C\underline{X},\underline{X})^{K-\{i\}} \longrightarrow \Sigma(C\underline{X},\underline{X})^K$ is given by the inclusion of the first wedge summand. Since this is a homotopy equivalence by assumption, it follows that $\Sigma^2|K_I| \wedge \widehat{X}^I$ must be contractible for every non-face $I\notin K$ containing the vertex $i$. As the $CW$-complexes $X_1,\ldots,X_m$ are all non-contractible, so are their smash products $\widehat{X}^I=\bigwedge_{j\in I}X_j$, so this implies in particular that $\Sigma^2|K_I|$ is contractible for every non-face $I\notin K$ containing $i$. 

To show that $K = \{i\} \ast (K-\{i\})$, it suffices to show that $\{i\} \cup \sigma \in K$ whenever $\sigma\in K$. First observe that $\{i,j\}\in K$ for all $j\in [m]$, since otherwise we would have $\Sigma^2|K_{\{i,j\}}|=\Sigma^2S^0\simeq S^2 \not\simeq \ast$, contradicting the conclusion of the previous paragraph. Next, assume inductively that $\{i\} \cup \sigma \in K$ for every simplex $\sigma\in K$ with $|\sigma|=n$. Let $\tau=\{j_1,\ldots,j_{n+1}\} \in K$. Then for each $1\le k\le n+1$, we have $\{j_1,\ldots,\widehat{j_k},\ldots,j_{n+1}\} \in K$, which implies $\{i,j_1,\ldots,\widehat{j_k},\ldots,j_{n+1}\} \in K$. Now every proper subset of $\{i,j_1,\ldots,j_{n+1}\}$ is a simplex of $K$, so if $\{i,j_1,\ldots,j_{n+1}\} \notin K$, then
\[ \Sigma^2|K_{\{i,j_1,\ldots,j_{n+1}\}}| = \Sigma^2\partial\Delta^{n+1} \simeq S^{n+2} \not\simeq \ast, \]
which is a contradiction. Therefore $\{i\}\cup\tau\in K$, which completes the induction. 
\end{proof}

By iterating Lemma \ref{key lemma}, we obtain the following simple combinatorial characterization of when the inclusion of a full subcomplex induces a homotopy equivalence of polyhedral products.

\begin{proposition} \label{TFAE}
Let $K$ be a simplicial complex on the vertex set $[m]$ and let $\underline{X}=\{X_j\}_{j=1}^m$ be a sequence of pointed $CW$-complexes which are non-contractible. For $I\subseteq [m]$, let $j_I \colon (C\underline{X}, \underline{X})^{K_I} \longrightarrow (C\underline{X}, \underline{X})^K$ be the natural map induced by the inclusion $K_I \subseteq K$. The following conditions are equivalent:
\begin{enumerate}
\item $j_I \colon (C\underline{X}, \underline{X})^{K_I} \longrightarrow (C\underline{X}, \underline{X})^K$ is a homotopy equivalence;
\item $\mathrm{core}(K) \subseteq K_I$;
\item $\mathrm{star}_K(v) = K$ for all $v \in [m]\backslash I$;
\item $\mathrm{link}_K(v) = K-\{v\}$ for all $v \in [m]\backslash I$; 
\item $K=\Delta^{m-|I|-1} \ast K_I$.
\end{enumerate}
\end{proposition}

\begin{proof}
The equivalence of conditions (b), (c), (d) and (e) follows immediately from the definitions.

(e) $\Rightarrow$ (a): This can be proved exactly as in the proof of Lemma \ref{key lemma}.

(a) $\Rightarrow$ (e): Write $[m]\backslash I=\{j_1,\ldots,j_p\}$ and note that the map $j_I$ factors as a composite of inclusions 
\[ (C\underline{X}, \underline{X})^{K-\{j_1,\ldots,j_p\}} \longrightarrow\cdots\longrightarrow (C\underline{X}, \underline{X})^{K-\{j_1\}}\longrightarrow (C\underline{X}, \underline{X})^K \]
where each map above has a left inverse by Proposition \ref{retract}. Therefore if $j_I$ is a homotopy equivalence, then so is each map in the composite, so by Lemma \ref{key lemma} we obtain
\begin{align*}
K &= \{j_1\}\ast(K-\{j_1\}) \\
&= \{j_1\}\ast\{j_2\}\ast(K-\{j_1,j_2\}) \\
\shortvdotswithin{=} 
&= \{j_1\}\ast\cdots\ast\{j_p\}\ast(K-\{j_1,\ldots,j_p\}) \\
&= \Delta^{m-|I|-1} \ast K_I,
\end{align*}
as desired.
\end{proof}

\section{Proofs of Theorems \ref{homeomorphic} and \ref{homotopy equivalent}}

In this section we restate and prove the main results and discuss some consequences. We begin with a lemma well known to homotopy theorists and include a short proof for completeness.

\begin{lemma} \label{wedge of spheres lemma}
If a space $Y$ is a homotopy retract of a simply-connected wedge of spheres $\bigvee_{\alpha \in \mathcal{I}} S^{n_\alpha}$, then $Y$ has the homotopy type of a wedge of spheres.
\end{lemma}

\begin{proof}
Suppose $Y$ is a homotopy retract of $\bigvee_{\alpha \in \mathcal{I}} S^{n_\alpha}$ where $n_\alpha \ge 2$ for all $\alpha \in \mathcal{I}$. Then there is a map $r \colon \bigvee_{\alpha \in \mathcal{I}} S^{n_\alpha} \longrightarrow Y$ inducing a split epimorphism in integral homology and the Hurewicz natural transformation gives a commutative diagram
\[
\xymatrix{
\pi_\ast(\bigvee_{\alpha\in \mathcal{I}} S^{n_\alpha}) \ar[r] \ar[d]^{r_\ast} & H_\ast(\bigvee_{\alpha\in \mathcal{I}} S^{n_\alpha}) \ar[d]^{r_\ast} \\
\pi_\ast(Y) \ar[r] & H_\ast(Y).
}
\]
The bottom horizontal arrow is an epimorphism since the top horizontal and right vertical ones are. By hypothesis, $H_\ast(Y)$ is a graded free abelian group, so by choosing Hurewicz pre-images of the elements of a basis for $H_\ast(Y)$ and taking their wedge sum we obtain a map from a wedge of spheres into $Y$ inducing an isomorphism in homology. This map is therefore a homotopy equivalence by Whitehead's Theorem since it also follows from the hypothesis that $Y$ has the homotopy type of a simply-connected $CW$-complex.
\end{proof}


{
\def\thetheorem{\ref{homeomorphic}}
\begin{theorem}
If $\mathcal{Z}_K$ is homeomorphic to a connected sum of sphere products with two spheres in each product, then $K$ is minimally non-Golod.
\end{theorem}
\addtocounter{theorem}{-1}
}

\begin{proof}
Suppose there is a homeomorphism \[ \mathcal{Z}_K \cong \operatorname*{\#}_{k=1}^\ell (S^{n_k} \times S^{n-n_k}) \]
where $\ell$ is finite and $3 \le n_k \le n-3$ for each $k=1,\ldots,\ell$ since every moment-angle complex is a finite $2$-connected $CW$-complex. Note that $H^*(\mathcal{Z}_K)$ has a non-trivial cup product, so $K$ is not Golod. Let $i\in [m]$ be a vertex of $K$ and let $j\colon\mathcal{Z}_{K-\{i\}} \longrightarrow \mathcal{Z}_K$ be the map induced by the inclusion $K-\{i\} \subseteq K$. It follows from the definition of a polyhedral product \eqref{pp definition} that any point $(z_1,\ldots,z_m) \in (D^2)^m$ with $|z_i|<1$ lies in $\mathcal{Z}_K$ outside the image of $j$, and hence $j$ is not surjective. Since $\mathcal{Z}_K$ is a closed manifold by assumption, the complement of a point in $\mathcal{Z}_K$ deformation retracts onto the $(n-1)$-skeleton of $\mathcal{Z}_K$. Therefore, up to homotopy, $j$ lifts through the $(n-1)$-skeleton of $\operatorname{\#}_{k=1}^\ell (S^{n_k} \times S^{n-n_k})$, which is $\bigvee_{k=1}^\ell(S^{n_k} \vee S^{n-n_k})$ since the connected sum of sphere products has the homotopy type of a wedge of spheres with a single top cell attached by a sum of Whitehead products of the form $w_k\colon S^{n-1} \longrightarrow S^{n_k} \vee S^{n-n_k}$. 

Combining the above observation with the fact that $j$ admits a retraction $r\colon \mathcal{Z}_K \longrightarrow \mathcal{Z}_{K-\{i\}}$ by Proposition \ref{retract}, we obtain a diagram
\[
\xymatrix{
 & \displaystyle\bigvee_{k=1}^\ell (S^{n_k} \vee S^{n-n_k})  \ar[d]  & \\
\mathcal{Z}_{K-\{i\}} \ar[r]^-j \ar[ur] \ar@{=}[dr]  &  \mathcal{Z}_K \ar[d]^r \ar@{^{(}->}[r]  &  (D^2)^m \ar[d]^{\mathrm{proj}} \\ 
 &  \mathcal{Z}_{K-\{i\}} \ar@{^{(}->}[r]  &  (D^2)^{m-1} 
}
\]
where the bottom triangle and square commute and the top triangle commutes up to homotopy. It follows that $\mathcal{Z}_{K-\{i\}}$ is a homotopy retract of $\bigvee_{k=1}^\ell(S^{n_k} \vee S^{n-n_k})$ and hence is homotopy equivalent to a wedge of spheres by Lemma \ref{wedge of spheres lemma}. Consequently, $K-\{i\}$ is Golod, which implies $K$ is minimally non-Golod as this holds for every vertex $i$ of $K$.
\end{proof}

Let $\mathcal{K}$ denote the collection of simplicial complexes whose corresponding moment-angle complexes are homeomorphic to connected sums of sphere products with two spheres in each product. Then $\mathcal{K}$ includes the nerve complexes of all simple polytopes obtained by vertex truncations of one or a product of two simplices and all even dimensional dual neighbourly polytopes, as well as all simplicial complexes obtained from these by applying the simplicial wedge construction or vertex truncation operations in any order (see \cite{GL} and \cite{CFW}).

\begin{corollary}
Let $K\in \mathcal{K}$. Then every proper full subcomplex $K_I$ of $K$ has the property that $\mathcal{Z}_{K_I}$ is homotopy equivalent to a wedge of spheres. In particular, the Stanley--Reisner ring $\mathbf{k}[K_I]$ is Golod over any ring $\mathbf{k}$.
\end{corollary}

It is not true that $K$ is minimally non-Golod whenever $\mathcal{Z}_K$ has the \emph{homotopy type} of a connected sum of two-fold products of spheres. We describe the smallest possible counterexample below before turning to the proof of Theorem \ref{homotopy equivalent}.

\begin{example} \label{example}
Consider the simplicial complex $K$ on $5$ vertices with facets $\{1,2,5\}$, $\{2,3,5\}$, $\{3,4,5\}$ and $\{1,4,5\}$. Observe that $K$ is the cone over the boundary of a square and can be written as the join $K=K_4 \ast \{5\}$. It is easy to see that $\mathcal{Z}_{K_4}\cong S^3\times S^3$. (More generally, if $K_m$ is the boundary of an $m$-gon with $m\ge 4$, then $\mathcal{Z}_{K_m}$ is homeomorphic to a connected sum of sphere products by \cite{BM}.) It follows that $\mathcal{Z}_K$ has the homotopy type of a connected sum of sphere products since
\[ \mathcal{Z}_K \cong \mathcal{Z}_{K_4} \times \mathcal{Z}_{\{5\}} \cong S^3\times S^3\times D^2 \simeq S^3\times S^3, \]
but $K$ is not minimally non-Golod since its deletion complex $K-\{5\}=K_4$ is not Golod. 
\end{example}

{
\def\thetheorem{\ref{homotopy equivalent}}
\begin{theorem}
If $\mathcal{Z}_K$ is homotopy equivalent to a connected sum of sphere products with two spheres in each product, then $K = \Delta^d \ast L$ for some $d\ge -1$ where $L$ is Gorenstein* and minimally non-Golod. 
\end{theorem}
\addtocounter{theorem}{-1}
}

\begin{proof}
Suppose there is a homotopy equivalence \[ \mathcal{Z}_K \simeq \operatorname*{\#}_{k=1}^\ell (S^{n_k} \times S^{n-n_k}) \]
for some $\ell\ge 1$ and $3 \le n_k \le n-3$ for each $k=1,\ldots,\ell$. For each vertex $i\in [m]$, consider the natural inclusion $j\colon\mathcal{Z}_{K-\{i\}} \longrightarrow \mathcal{Z}_K$ and the induced homomorphism 
\[ j^\ast \colon \mathbb{Z}\cong H^n(\mathcal{Z}_K) \longrightarrow H^n(\mathcal{Z}_{K-\{i\}}). \]
By Proposition \ref{retract}, $j^\ast$ has a right inverse, so either $j^\ast$ is an isomorphism or else $H^n(\mathcal{Z}_{K-\{i\}})=0$. If $j^\ast$ is an isomorphism, then the Poincar\'e duality of $H^\ast(\mathcal{Z}_K)$ implies that $j$ induces an isomorphism in cohomology in all dimensions and is thus a homotopy equivalence. In this case, we obtain that $K=\{i\}\ast(K-\{i\})$ by Lemma~\ref{key lemma}. It follows that the set of all vertices $i\in [m]$ for which the map $j^\ast$ above is an isomorphism span a simplex $\Delta^d$ in $K$ and that $K=\Delta^d \ast L$, where $L$ is the full subcomplex of $K$ on the set of vertices $i\in [m]$ for which $j\colon\mathcal{Z}_{K-\{i\}} \longrightarrow \mathcal{Z}_K$ is not a homotopy equivalence. (Note that $L=\mathrm{core}(K)$ by Proposition \ref{TFAE}, and that $-1\le d\le m-5$ since $L$ is a simplicial complex on $m-d-1$ vertices and $\mathcal{Z}_L\simeq\mathcal{Z}_K$ cannot have the homotopy type of a connected sum of two-fold products of spheres if $L$ has less than $4$ vertices.)

For each vertex $i$ of $L$, we have that $H^n(\mathcal{Z}_{L-\{i\}})=0$. Since $\mathcal{Z}_{L-\{i\}}$ is a retract of $\mathcal{Z}_L\simeq\mathcal{Z}_K \simeq \operatorname*{\#}_{k=1}^\ell (S^{n_k} \times S^{n-n_k})$, it follows that $\mathcal{Z}_{L-\{i\}}$ has the homotopy type of a simply-connected $CW$-complex of dimension less than $n$. Therefore the map $\mathcal{Z}_{L-\{i\}} \longrightarrow \mathcal{Z}_L$ lifts up to homotopy through the $(n-1)$-skeleton of $\mathcal{Z}_L\simeq \operatorname*{\#}_{k=1}^\ell (S^{n_k} \times S^{n-n_k})$. The same argument as in the proof of Theorem \ref{homeomorphic} now shows that $\mathcal{Z}_{L-\{i\}}$ is homotopy equivalent to a wedge of spheres and hence that $L-\{i\}$ is Golod.

Finally, it follows from the Poincar\'e duality of $H^\ast(\mathcal{Z}_K)\cong \mathrm{Tor}_{\mathbb{Z}[v_1,\ldots,v_m]}^\ast(\mathbb{Z}[K],\mathbb{Z})$ that $K$ is a Gorenstein complex (see \cite[Theorem~4.6.8]{BP}). Thus $L=\mathrm{core}(K)$ is a Gorenstein* complex.
\end{proof}

\begin{remark}
A combinatorial-topological characterization due to Stanley \cite{S} states that a simplicial complex $K$ is Gorenstein* if and only if $K$ is a generalized homology sphere. In \cite{C}, it was shown that a moment-angle complex $\mathcal{Z}_K$ is a closed topological manifold of dimension $m+n$ if and only if $K$ is a generalized homology $(n-1)$-sphere. In particular, any moment-angle complex satisfying the hypothesis of Theorem \ref{homotopy equivalent} is in fact homeomorphic to a product of disks and a closed orientable manifold with the homotopy type of a connected sum of sphere products. 
\end{remark}

\section{An analogue for real moment-angle complexes}

In this section we prove an analogue of Theorem \ref{homotopy equivalent} for real moment-angle complexes. Recall that the real moment-angle complex corresponding to $K$ is defined by the polyhedral product $\mathcal{R}_K=(C\underline{X},\underline{X})^K$ for the sequence $\underline{X}=\{X_i\}_{i=1}^m$ with $X_i=S^0$ for each $i=1,\ldots,m$.

\begin{example}
Let $K_m$ be the boundary of an $m$-gon. If $m\ge 4$, then the corresponding real moment-angle complex $\mathcal{R}_{K_m} \cong \operatorname*{\#}_{k=1}^g(S^1\times S^1)$ is an orientable surface of genus $g=1+(m-4)2^{m-3}$ by a result attributed to Coxeter (see \cite[Proposition~4.1.8]{BP}). In this case, each deletion complex $K_m-\{i\}$ is a path graph which is Golod and $\mathcal{R}_{K_m-\{i\}}$ is homotopy equivalent to a wedge of circles. 
\end{example}

As the example above illustrates, real moment-angle complexes need not be simply-connected. For this reason, we will need a stronger version of Lemma~\ref{wedge of spheres lemma}. A proof that the statement of Lemma~\ref{wedge of spheres lemma} still holds without the simply-connectedness hypothesis, provided that the index set $\mathcal{I}$ is finite, is given in \cite[Theorem~3.3]{MMM}. 

A further modification to the proof of Theorem \ref{homotopy equivalent} is required to relate the homotopy type of $\mathcal{R}_K$ to the homotopy type of $\mathcal{Z}_K$ and hence to the Golodness of $K$. For this, we refer to the work of Iriye and Kishimoto \cite{IK2} on the fat wedge filtration of $\mathcal{R}_K$ and its relation to the homotopy type of polyhedral products of the form $(C\underline{X},\underline{X})^K$.

\begin{theorem} \label{real}
If $\mathcal{R}_K$ is homotopy equivalent to a connected sum of sphere products with two spheres in each product, then $K = \Delta^d \ast L$ for some $d\ge -1$ where $L$ is minimally non-Golod.
\end{theorem}

\begin{proof}
Suppose $\mathcal{R}_K$ is homotopy equivalent to a connected sum of sphere products $\operatorname*{\#}_{k=1}^\ell (S^{n_k} \times S^{n-n_k})$ with $\ell\ge 1$ and $1 \le n_k \le n-1$ for each $k=1,\ldots,\ell$. As in the proof of Theorem \ref{homotopy equivalent}, $K=\Delta^d\ast L$ where $L=\mathrm{core}(K)$ has the property that $j\colon\mathcal{R}_{L-\{i\}} \longrightarrow \mathcal{R}_L$ is not a homotopy equivalence for any vertex $i$ of $L$ by Proposition \ref{TFAE}. A priori, this does not immediately imply that $j$ does not induce an isomorphism in cohomology since $\mathcal{R}_L$ and its retract $\mathcal{R}_{L-\{i\}}$ are not necessarily simply-connected. However, the proof of the forward implication in Lemma \ref{key lemma} shows that if $\Sigma j\colon \Sigma\mathcal{R}_{L-\{i\}} \longrightarrow \Sigma\mathcal{R}_L$ is a homotopy equivalence, then $L=\{i\}\ast (L-\{i\})$, contradicting that $\{i\} \in L=\mathrm{core}(L)$. Thus $\Sigma j$ is not a homotopy equivalence, which implies that $\Sigma j$ does not induce an isomorphism in cohomology since the suspensions $\Sigma\mathcal{R}_{L-\{i\}}$ and $\Sigma\mathcal{R}_L$ are simply-connected. It follows from the Poincar\'e duality of $H^\ast(\mathcal{R}_L)\cong H^\ast(\operatorname*{\#}_{k=1}^\ell (S^{n_k} \times S^{n-n_k}))$ that 
\[ j^\ast\colon \mathbb{Z} \cong H^n(\mathcal{R}_L) \longrightarrow H^n(\mathcal{R}_{L-\{i\}}) \] 
is not an isomorphism, and hence the retract $\mathcal{R}_{L-\{i\}}$ does not contain the top cell of $\mathcal{R}_L$. Since $\mathcal{R}_{L-\{i\}}$ is then a homotopy retract of $\bigvee_{k=1}^\ell (S^{n_k} \times S^{n-n_k})$, we conclude that $\mathcal{R}_{L-\{i\}}$ is homotopy equivalent to a wedge of spheres by \cite[Theorem~3.3]{MMM}.

In \cite{IK2}, the fat wedge filtration of a real moment-angle complex is shown to be a cone decomposition. Since for each vertex $i$ of $L$, $\mathcal{R}_{L-\{i\}}$ is homotopy equivalent to a wedge of spheres, it follows that the attaching maps in this cone decomposition for $\mathcal{R}_{L-\{i\}}$ are null homotopic and by \cite[Theorem 1.2]{IK2}, the decomposition of $\Sigma (C\underline{X},\underline{X})^{L-\{i\}}$ in Theorem \ref{BBCG splitting} desuspends for any $\underline{X}$. In particular, $\mathcal{Z}_{L-\{i\}}$ is a suspension, which implies that $L-\{i\}$ is Golod (see \cite[Proposition~6.5]{IK2}). 
\end{proof}

\begin{corollary}
If $\mathcal{R}_K$ is homeomorphic to a connected sum of sphere products with two spheres in each product, then $K$ is minimally non-Golod.
\end{corollary}

\begin{proof}
Under the given assumption, $K=\Delta^d \ast L$ for some minimally non-Golod complex $L$ by Theorem \ref{real}. Therefore, $\mathcal{R}_K\cong \mathcal{R}_{\Delta^d} \times \mathcal{R}_L \cong D^{d+1} \times \mathcal{R}_L$. But since $\mathcal{R}_K$ is a manifold without boundary by assumption, the disk $D^{d+1}$ must have dimension $0$, so $d=-1$ and $K=L$ is minimally non-Golod.
\end{proof}

\end{document}